\documentclass[12pt,letterpaper,twoside]{article}

\usepackage{amsmath,amssymb}
\usepackage{amsthm}
\usepackage{stmaryrd}
\usepackage{fancyhdr}
\usepackage{times}
\usepackage{mathrsfs} 
\usepackage{titlesec}
\usepackage[pdftex,dvips]{geometry}

\newtheoremstyle{fact}
     {\topsep}
     {\topsep}
     {\slshape}
     {}
     {\bfseries}
     {}
     { }
     {\thmname{#1}\thmnumber{ #2.}\thmnote{ \rm (#3)}}

\newtheorem{theorem}{Theorem}[section]
\newtheorem{Ltheorem}{Theorem}

\newtheorem*{theorem*}{Theorem} 
\newtheorem{lemma}[theorem]{Lemma}
\newtheorem{proposition}[theorem]{Proposition}
\newtheorem{corollary}[theorem]{Corollary}

\theoremstyle{definition}
\newtheorem{definition}[theorem]{Definition}
\newtheorem{remark}[theorem]{Remark}

\newtheorem*{remark*}{Remark}

\newtheorem{notation}[theorem]{Notation}

\newtheorem*{question*}{Question}
\newtheorem*{examples*}{Examples}  

\newtheorem*{example*}{Example}

\newtheorem*{convention*}{Convention}

\theoremstyle{fact}

\newtheorem{ftheorem}[theorem]{Theorem}

\newtheorem{fproposition}[theorem]{Proposition}

\newenvironment{myromanlist}[1][enumi]{\begin{list}{{\rm (\roman{#1})}}
{\usecounter{#1}\setlength{\labelwidth}{25pt}\setlength{\topsep}{-6pt}
\setlength{\itemsep}{-4pt} \setlength{\leftmargin}{25pt}}}{\end{list}}

\newenvironment{myalphlist}[1][enumi]{\begin{list}{{\rm (\alph{#1})}}
{\usecounter{#1}\setlength{\labelwidth}{25pt}\setlength{\topsep}{-6pt}
\setlength{\itemsep}{-4pt} \setlength{\leftmargin}{25pt}}}{\end{list}}

\def\proofont{\fontseries{bx}\fontshape{sc}\selectfont}
\def\proofname{Proof.}

\newcommand{\Note}[1]{}

\makeatletter
\renewenvironment{proof}[1][\proofname]{\par
  \normalfont
  \topsep6\p@\@plus6\p@ \trivlist
  \item[\hskip\labelsep\noindent\proofont #1]\ignorespaces
}{%
  \qed\endtrivlist
}
\makeatother

\titlelabel{\thetitle.\ }
\titleformat*{\section}{\normalsize\bfseries\centering}
\titleformat*{\subsection}{\normalsize\bfseries}
\titlespacing{\subsection}{0pt}{\topsep}{0.5ex}
\titleformat{\subsection}[runin]{\normalfont\bfseries}{%
\thesubsection.}{0.5ex}{}[.]

\author{D. Dikranjan\thanks{The first author acknowledges 
the financial aid received from MCYT, MTM2006-02036 and FEDER funds.} 
{ }and G\'abor Luk\'acs\thanks{The second author gratefully acknowledges 
the generous financial  support received from NSERC and the University of 
Manitoba, which enabled him to do this research.}}
   
\title{Quasi-convex sequences\\ in the circle and the $3$-adic 
integers\thanks{{\em 2000 Mathematics Subject Classification}: 
Primary 22A05; Secondary  22B05.\endgraf \hspace{5.5pt} 
{\em Keywords:} quasi-convex sets, qc-dense sets, torus group, $3$-adic 
integers}}

\begin{document}

\makeatletter
\def\@fnsymbol#1{\ifcase#1\or * \or 1 \or 2  \else\@ctrerr\fi\relax}

\let\mytitle\@title
\chead{\small\itshape D. Dikranjan and G. Luk\'acs / 
Quasi-convex sequences in the circle and the $3$-adic integers}
\fancyhead[RO,LE]{\small \thepage}
\makeatother

\maketitle

\def\thanks#1{} 

\thispagestyle{empty}


\begin{abstract}
In this paper, we present families of quasi-convex sequences converging to 
zero in the circle group $\mathbb{T}$, and  the group $\mathbb{J}_3$ of 
$3$-adic integers. These sequences are determined by an increasing 
sequences of integers. For an increasing sequence 
$\underline{a}=\{a_n\}_{n=0}^\infty \subseteq \mathbb{Z}$,
put $g_n=a_{n+1}-a_n$. We prove that 

\begin{myalphlist}

\item
the set 
$\{0\}\cup\{\pm 3^{-(a_n+1)} \mid  n\in \mathbb{N}\} $ is quasi-convex in 
$\mathbb{T}$ if and only if $a_0>0$ and  $g_n>1$ for every $n\in \mathbb{N}$;

\item
the set $\{0\}\cup\{\pm 3^{a_n} \mid n\in \mathbb{N}\} $ is quasi-convex 
in the group $\mathbb{J}_3$ of $3$-adic integers if and only if 
$g_n>1$ for every $n\in \mathbb{N}$.
\vspace{3pt}

\end{myalphlist}
Moreover, we solve an open problem from \cite{DikLeo} by 
providing a 
complete characterization of the sequences $\underline{a}$ such that
\mbox{$\{0\}\cup\{\pm 2^{-(a_n+1)}\mid n\in \mathbb{N}\}$} is quasi-convex 
in~$\mathbb{T}$. Using this result, we also obtain a characterization 
of the sequences $\underline{a}$ such  that the set 
\mbox{$\{0\}\cup\{\pm 2^{-(a_n+1)}\mid n\in \mathbb{N}\}$} is quasi-convex 
in $\mathbb{R}$.
\end{abstract}

\section{Introduction}

\label{sect:intro}

One of the mains sources of inspiration for the theory of topological 
groups is the theory of topological vector spaces, where the notion of 
convexity plays a prominent role. In this context, the reals $\mathbb{R}$ 
are replaced with the circle group $\mathbb{T}=\mathbb{R}/\mathbb{Z}$, and 
linear functionals are replaced by {\em characters}, that is, continuous 
homomorphisms to $\mathbb{T}$. By making substantial use of characters, 
Vilenkin introduced  the notion of quasi-convexity for abelian topological 
groups as a  counterpart of convexity in topological vector spaces (cf.  
\cite{Vilenkin}).  The counterpart of locally convex spaces are the 
locally quasi-convex groups.  This class includes all locally compact 
abelian groups and locally convex topological vector spaces 
(cf.~\cite{Banasz}). Every locally quasi-convex group is {\em  
maximally  almost periodic} (briefly, {\em MAP}), which means that its 
characters separate its points.

Unlike the ``geometrically" transparent property of convexity,
quasi-convexity remains an admittedly mysterious property.  Although
locally quasi-convex groups have been studied by many authors
(cf.~\cite{Aussenhofer}, \cite{Banasz}, and \cite{BruPhD}), their work did
not completely reveal the nature of the small quasi-convex sets.
Aussenhofer proved that the quasi-convex hull of a finite subset of a MAP
group is finite (cf.~\cite{Aussenhofer}), 
and a slightly more general
property was shown in \cite{DikKun} (many examples of finite quasi-convex
sets are found in \cite{BeiLeoDikSte} and \cite{deLeoPhD}). Since finite
sets are compact, the following statement---also known as the {\em
quasi-convex compactness property}---established by Hern\'andez, and
independently, by Bruguera and Mart\'{\i}n-Peinador, can be considered
a generalization of Aussenhofer's theorem:  A~metrizable locally
quasi-convex abelian group $G$ is complete if and only if the quasi-convex 
hull of every compact subset of $G$ is compact (cf.~\cite{Hern2}
and~\cite{BrugMar2}). Luk\'acs extended this result, and proved that
a~metrizable abelian group $A$ is MAP and has the quasi-convex compactness 
property if and only if it is locally quasi-convex and complete;
he also showed that such groups are characterized by the property 
that the  evaluation map $\alpha_A\colon A \rightarrow \hat{\hat A}$ is a
closed embedding (cf.~\cite[I.34]{GLdualtheo}).

Interest in the compact quasi-convex sets stems from the theory of 
Mackey groups (cf.~\cite{ChaMarTar} and~\cite{deLeoPhD}). The paper 
\cite{DikLeo} is dedicated to the study of countably infinite quasi-convex 
sets (in fact, sequences that converge to zero) in some familiar locally 
compact abelian groups such as $\mathbb{R}$, and the compact groups 
$\mathbb{T}$ and $\mathbb{J}_2$ ($2$-adic integers). In order to formulate 
the main results of \cite{DikLeo}, we introduce some notations.

Let $\pi\colon \mathbb{R} \rightarrow \mathbb{T}$  denote  the canonical 
projection. Since the restriction 
$\pi_{|[0,1)}\colon [0,1) \rightarrow \mathbb{T}$ is
a bijection, we often identify in the sequel, {\em par abus de language},  
a number $a\in [0,1)$ with its image (coset) $\pi(a)=a+\mathbb{Z}\in 
\mathbb{T}$.  We put $\mathbb{T}_m:=\pi([-\frac{1}{4m},\frac{1}{4m}])$ 
for all $m\in \mathbb{N}\backslash\{0\}$. According to standard 
notation in this area, we use $\mathbb{T}_+$ to denote $\mathbb{T}_1$.
For an abelian topological group $G$, we denote by $\widehat{G}$ 
the {\em Pontryagin dual} of a $G$, that is, the group of all 
characters of $G$, endowed with the compact-open topology.

\begin{definition}\label{def:into:qc}
For $E\subseteq G$ and $A \subseteq \widehat{G}$,  the 
{\em polars} of $E$ and $A$ are defined as
\begin{align}
E^\triangleright=\{\chi\in \widehat{G} \mid \chi(E) \subseteq \mathbb{T}_+\}
\quad \text{and} \quad
A^\triangleleft=\{ x \in A \mid  \forall \chi \in A, 
\chi(x) \in \mathbb{T}_+ \}.
\end{align}
The set $E$ is said to be {\em quasi-convex} if 
$E=E^{\triangleright\triangleleft}$.
\end{definition}

Obviously, \mbox{$E\subseteq E^{\triangleright\triangleleft}$} holds for 
every  \mbox{$E\subseteq G$}. Thus, $E$ is quasi-convex
if and only if for every \mbox{$x\in G\backslash E$} there exists 
\mbox{$\chi\in E^\triangleright$} such that 
\mbox{$\chi(x)\not\in \mathbb{T}_+$}. 
The set  $Q_{G}(E):=E^{\triangleright\triangleleft}$ is the smallest 
quasi-convex set of $G$ that contains $E$, and it is 
called the {\em quasi-convex hull} of $E$. In what follows, we rely on
the following general property of quasi-convexity.

\begin{fproposition}[{\cite[I.3(e)]{GLdualtheo}, \cite[2.7]{DikLeo}}]
\label{prop:intro:qc-hom}
If $f\colon G\to H$ is a continuous homomorphism of abelian
topological groups and $E\subseteq G$, then $f(Q_G(E))\subseteq Q_H(f(E))$.
\end{fproposition}

Dikranjan and de Leo proved the following theorem.

\begin{samepage}
\begin{ftheorem} \label{thm:intro:DikLeo}
Let $\underline a= \{a_n\}_{n=0}^\infty$ be an increasing sequence of 
integers, and put $g_n = a_{n+1} -a_n$.

\begin{myalphlist}

\item
{\rm (\cite[1.1]{DikLeo})}
If $a_0>0$ and $g_n>1$ for every $n\in \mathbb{N}$, then
$K_{\underline{a},2}= \{0\}\cup\{\pm 2^{-(a_n+1)}\mid 
n\in \mathbb{N}\}$ is quasi-convex in $\mathbb{T}$.

\item
{\rm (\cite[1.2]{DikLeo})}
If $g_n>1$ for every $n\in \mathbb{N}$, then 
$R_{\underline{a},2}= \{0\}\cup\{\pm 2^{-(a_n+1)}\mid n\in \mathbb{N}\}$ 
is quasi-convex in $\mathbb{R}$.

\item
{\rm (\cite[1.4]{DikLeo})}
If $a_0\geq 0$ and $g_n>1$ for every $n\in \mathbb{N}$, then the set 
$L_{\underline{a},2}:=\{0\}\cup \{\pm 2^{a_n}\mid n\in \mathbb{N}\}$
is quasi-convex in $\mathbb{J}_2$.

\end{myalphlist}
\end{ftheorem}
\end{samepage}

The question of whether the conditions formulated in 
Theorem~\ref{thm:intro:DikLeo}(a) are also necessary in order to assure 
that $K_{\underline{a},2}$ is quasi-convex was left open by Dikranjan and 
de Leo, and in fact, was formulated  as an open problem 
(cf.~\cite[5.1(a)]{DikLeo}). One of the main results of this paper is the 
following theorem, which provides a complete solution to this open 
problem.

\begin{samepage}
\begin{Ltheorem} \label{thm:T2:main}
Let $\underline a= \{a_n\}_{n=0}^\infty$ be an increasing sequence of 
non-negative
integers, and put \mbox{$x_n = 2^{-(a_n+1)}$}. The set
$K_{\underline a,2} = \{0\} \cup \{\pm x_n \mid n \in \mathbb{N}\}$
is quasi-convex in $\mathbb{T}$ if and only if  the following three
conditions hold:

\begin{myromanlist}

\item
$a_0 > 0$;

\item
the gaps $g_n=a_{n+1}-a_n$ satisfy 
$g_n > 1$ for all but possibly one index $n\in \mathbb{N}$;

\item
if $g_n =1$ for some $n\in \mathbb{N}$, then $g_{n+1} > 2$.

\end{myromanlist}
\end{Ltheorem}
\end{samepage}

Theorem~\ref{thm:T2:main} also yields a complete 
characterization of the sequences $\underline{a}$ such that
$R_{\underline{a},2}$ is quasi-convex.

\begin{samepage}
\begin{Ltheorem} \label{thm:R2:main}
Let $\underline a= \{a_n\}_{n=0}^\infty$ be an increasing sequence of
integers, and put \mbox{$r_n = 2^{-(a_n+1)}$}. The set
$R_{\underline a,2} = \{0\} \cup \{\pm r_n \mid n \in \mathbb{N}\}$
is quasi-convex in $\mathbb{R}$ if and only if the following two
conditions hold:

\begin{myromanlist}

\item
the gaps $g_n=a_{n+1}-a_n$ satisfy
$g_n > 1$ for all but possibly one index $n\in\mathbb{N}$;

\item
if $g_n =1$ for some $n\in \mathbb{N}$, then $g_{n+1} > 2$. 

\end{myromanlist}
\end{Ltheorem}
\end{samepage}

It turns out that by replacing the prime $2$ with $3$ in parts (a) and (c) 
of Theorem~\ref{thm:intro:DikLeo}, one obtains conditions that are not 
only sufficient, but also necessary.

\begin{Ltheorem} \label{thm:T3:main}
Let $\underline a= \{a_n\}_{n=0}^\infty$ be an increasing sequence of 
non-negative integers, and put 
\mbox{$x_n = 3^{-(a_n+1)}$}. The set
$K_{\underline a,3} = \{0\} \cup \{\pm x_n \mid n \in \mathbb{N}\}$
is quasi-convex in $\mathbb{T}$ if and only if 
the following two 
conditions 
hold:

\begin{myromanlist}

\item
$a_0 > 0$;

\item
the  gaps $g_n=a_{n+1}-a_n$ satisfy $g_n > 1$ for all $n\in \mathbb{N}$.

\end{myromanlist}
\end{Ltheorem}

\begin{Ltheorem} \label{thm:J3:main}
Let $\underline a= \{a_n\}_{n=0}^\infty$ be an increasing sequence of 
non-negative integers, and put \mbox{$y_n = 3^{a_n}$}. The set
$L_{\underline a,3} = \{0\} \cup \{\pm y_n \mid n \in \mathbb{N}\}$
is quasi-convex in $\mathbb{J}_3$ if and only if the gaps
$g_n=a_{n+1}-a_n$ satisfy the condition  $g_n > 1$ for all $n\in \mathbb{N}$.
\end{Ltheorem}

In spite of the apparent similarity between the Main Lemma of 
\cite{DikLeo} and Theorems~\ref{thm:T3:Q12} and~\ref{thm:J3:Q12},
our techniques for proving Theorems~\ref{thm:T3:main} 
and~\ref{thm:J3:main} do differ from the ones used by Dikranjan and de 
Leo to prove Theorem~\ref{thm:intro:DikLeo} (a) and (c) 
(cf.~\cite{DikLeo}).  We believe 
that the techniques presented here can also be used to prove appropriate 
generalizations of Theorems~\ref{thm:T3:main} and~\ref{thm:J3:main} for 
arbitrary primes $p > 2$, at least as far as sufficiency of the conditions 
is concerned.

The paper is structured as follows. In \S\ref{sect:nec}, we present
conditions that are necessary for the quasi-convexity of a sequence 
converging to zero. These results are used later on, in the proofs 
of the necessity of the conditions in
Theorems~\ref{thm:T2:main}, \ref{thm:R2:main}, and~\ref{thm:T3:main}.
In order to cover all three cases, we consider sequences of rationals 
of the form $\tfrac 1 {b_n}$ in $\mathbb{R}$, where $b_n \mid b_{n+1}$ for 
every $n \in \mathbb{N}$, as well as their images under $\pi$ in 
$\mathbb{T}$. \S \ref{sect:T2R2} is dedicated to the proof of 
Theorems~\ref{thm:T2:main} and~\ref{thm:R2:main}, the latter being an easy 
consequence of the earlier. In \S\ref{sect:T3}, we prove 
Theorem~\ref{thm:T3:main}, and in \S\ref{sect:J3},
the proof of Theorem~\ref{thm:J3:main} is presented.


\section{Necessary conditions in $\boldsymbol{\mathbb{T}}$ and 
$\boldsymbol{\mathbb{R}}$ for quasi-convexity}

\label{sect:nec}

Let $\{b_n\}_{n=0}^\infty$ be an increasing sequence of natural numbers 
such that $b_n | b_{n+1}$ for every $n \in \mathbb{N}$, and $b_0 > 1$.
In this section, we are concerned with finding conditions that are 
necessary for
\begin{align}
X & = \{0\}\cup \left\{\left. \pm \dfrac 1 {b_n} \right|  n \in 
\mathbb{N}\right\}  \subseteq \mathbb{T} \quad \text{and} \quad 
S  = \{0\}\cup \left\{\left. \pm \dfrac 1 {b_n} \right|  n \in 
\mathbb{N}\right\}\subseteq \mathbb{R}
\end{align}
to be quasi-convex in $\mathbb{T}$ and $\mathbb{R}$, respectively.
We put $q_n=\dfrac{b_{n+1}}{b_n}$ for each $n \in \mathbb{N}$.

\begin{theorem} \label{thm:nec:specR}
Suppose that $S$ is quasi-convex in $\mathbb{R}$. Then:

\begin{myalphlist}

\item
$|\{n \in \mathbb{N} \mid q_n=2 \}| \leq 1$;

\item
if $q_n=2$ for some $n\in \mathbb{N}$, then $q_{n+1}> 4$.

\end{myalphlist}
\end{theorem}

For the sake of transparency, we break up the proof of 
Theorem~\ref{thm:nec:specR} into two lemmas, the first of 
which holds in every abelian topological group.

\begin{lemma} \label{lemma:nec:h12}
Let $G$ be an abelian topological group, and $h,h_1,h_2 \in G$. Then:
\begin{myalphlist}

\item
$h_1 \pm h_2 \in Q_G(\{h_1,2h_1,h_2,2h_2\})$;

\item
$4h \in  Q_G(\{h,3h,6h\})$;

\item
$5h \in  Q_G(\{h,4h,8h\})$.

\end{myalphlist}
\end{lemma}

\begin{proof}
(a) Let $\chi \in \{h_1,2h_1,h_2,2h_2\}^\triangleright$.
Then $\chi(h_1)\in \mathbb{T}_+$ and 
$2\chi(h_1)=\chi(2h_1) \in \mathbb{T}_+$, and so
$\chi(h_1) \in \mathbb{T}_2$. Similarly, $\chi(h_2) \in \mathbb{T}_2$.
Thus, 
\begin{align}
\chi(h_1\pm h_2)=\chi(h_1)\pm\chi(h_2) \in \mathbb{T}_2 + \mathbb{T}_2 
=\mathbb{T}_+.
\end{align}
Hence, $h_1\pm h_2 \in  Q_G(\{h_1,2h_1,h_2,2h_2\})$.

(b) Let $\chi \in \{h,3h,6h\}^\triangleright$.
Then $\chi(h)\in \mathbb{T}_+$, 
$3\chi(h) = \chi(3h) \in \mathbb{T}_+$, and
$6\chi(h) = \chi(6h) \in \mathbb{T}_+$.
Thus, $\chi(h)\in\{1,3,6\}^\triangleright = \mathbb{T}_6$, and so
$\chi(4h)=4\chi(h) \in 4\mathbb{T}_6 \subseteq \mathbb{T}_+$.
Hence, $4h \in Q_G( \{h,3h,6h\})$.

(c) Let $\chi \in \{h,4h,8h\}^\triangleright$.  
Then $\chi(h)\in \mathbb{T}_+$,
$4\chi(h) = \chi(4h) \in \mathbb{T}_+$, and
$8\chi(h) = \chi(8h) \in \mathbb{T}_+$. Thus,
\begin{align}
\chi(h) \in \{1,4,8\}^\triangleright \in \mathbb{T}_8 \cup 
(\tfrac {15}{64} + \mathbb{T}_{16}) \cup 
(-\tfrac {15}{64} + \mathbb{T}_{16}).
\end{align}
Therefore,
\begin{align}
\chi(5h)=5\chi(h) \in 
5\mathbb{T}_8 \cup  
(\tfrac {11}{64} + 5\mathbb{T}_{16}) \cup
(-\tfrac {11}{64} + 5\mathbb{T}_{16}) \subseteq \mathbb{T}_+. 
\end{align}
Hence, $5h \in Q_G(\{h,4h,8h\})$, as desired. 
\end{proof}

\begin{lemma} \label{lemma:nec:b12}
If $\dfrac{1}{b_{n_1}} + \dfrac{1}{b_{n_2}} \in S$, then
$n_1=n_2$.
\end{lemma}

\begin{proof}
Suppose that $\frac{1}{b_{n_1}} + \frac{1}{b_{n_2}}\in S$, and so 
there is  $n_0 \in \mathbb{N}$ such that
\begin{align} \label{eq:nec:b12}
\frac{1}{b_{n_1}} + \frac{1}{b_{n_2}} = \frac{1}{b_{n_0}}.
\end{align}
Since the sequence $\{b_n\}_{n=0}^\infty$ is increasing,
the sequence $\{\frac 1 {b_n} \}_{n=0}^\infty$ is decreasing.
Thus, $n_0 < n_1, n_2$, because
$\frac{1}{b_{n_0}} > \frac{1}{b_{n_1}}, \frac{1}{b_{n_2}}$. Consequently,
$b_{n_0} \mid b_{n_1}, b_{n_2}$, and for $a=\frac{b_{n_1}}{b_{n_0}}$
and $b=\frac{b_{n_2}}{b_{n_0}}$, equation (\ref{eq:nec:b12}) can be 
rewritten as
\begin{align}
\frac 1 a + \frac 1 b =1,
\end{align}
where $a,b \in \mathbb{N}$. Therefore, $a=b=2$, which implies that
$b_{n_1}=b_{n_2}$. Hence, $n_1=n_2$.
\end{proof}

\begin{proof}[Proof of Theorem~\ref{thm:nec:specR}]
(a) Suppose that $q_{n_1}=q_{n_2}=2$. Put 
$h_1 =\frac{1}{b_{n_1+1}}$ and $h_2 =\frac{1}{b_{n_2+1}}$.
Then $2h_1 =\frac{1}{b_{n_1}}$ and 
$2h_2 =\frac{1}{b_{n_2}}$. Thus,
$\{h_1,2h_1,h_2,2h_2\} \subseteq S$, and so by 
Lemma~\ref{lemma:nec:h12}(a),
\begin{align}
h_1 + h_2 \in Q_{\mathbb{T}}(\{h_1,2h_1,h_2,2h_2\})
\subseteq Q_{\mathbb{T}}(S).
\end{align}
Consequently, 
$\frac{1}{b_{n_1+1}}+ \frac{1}{b_{n_2+1}}=h_1+h_2 \in S$, because $S$ is 
quasi-convex. Therefore, by Lemma~\ref{lemma:nec:b12},
$n_1+1=n_2+1$. Hence, $n_1=n_2$, as desired.

(b) Suppose that $q_n=2$.  For $h=\frac{1}{b_{n+2}}$, one has
$q_{n+1} h =\frac{1}{b_{n+1}}$ and $2q_{n+1} h =\frac{1}{b_{n}}$.
Thus, $\{h,q_{n+1}h,2q_{n+1}h\} \subseteq S$, and since $S$ is 
quasi-convex,
$Q_{\mathbb{T}}(\{h,q_{n+1}h,2q_{n+1}h\}) \subseteq S$.
If $q_{n+1}=3$,  then by Lemma~\ref{lemma:nec:h12}(b),
\begin{align}
(q_{n+1}+1)h = 4h \in  Q_{\mathbb{T}}(\{h,3h,6h\})=
Q_{\mathbb{T}}(\{h,q_{n+1}h,2q_{n+1}h\}) \subseteq S.
\end{align}
If $q_{n+1}=4$, then by Lemma~\ref{lemma:nec:h12}(c), 
\begin{align}
(q_{n+1}+1)h = 5h \in  Q_{\mathbb{T}}(\{h,4h,8h\})=
Q_{\mathbb{T}}(\{h,q_{n+1}h,2q_{n+1}h\}) \subseteq S.
\end{align}
In either case, 
\begin{align}
\frac{1}{b_{n+1}} + \frac{1}{b_{n+2}}=
q_{n+1}h + h =
(q_{n+1}+1)h \in S.
\end{align}
By Lemma~\ref{lemma:nec:b12}, this implies that $n+1=n+2$. This 
contradiction shows that $q_{n+1}\neq 3$ and $q_{n+1}\neq 4$.
Finally, we note that by (a), $q_{n+1} \neq 2$. Hence,
$q_{n+1} > 4$, as desired.
\end{proof}

In order to prove the analogue of Theorem~\ref{thm:nec:specR} for 
quasi-convexity of $X$ in $\mathbb{T}$, we first establish a general 
result relating quasi-convexity of subsets of $\mathbb{R}$ to 
quasi-convexity of their projections in~$\mathbb{T}$.


\begin{remark} \label{rem:nec:pi}
The map $\pi_{|(-\frac 1 2,\frac 1 2)}\colon (-\frac 1 2 ,\frac 1 2) 
\rightarrow \mathbb{T}\backslash\{\frac 1 2\}$ is a homeomorphism. Indeed, 
it is continuous and open because $\pi$ is so, and it is clearly 
surjective, so it remains to be seen that it is injective. Suppose that
$\pi(y_1)=\pi (y_2)$, where $y_1,y_2 \in (-\frac 1 2, \frac 1 2)$.
Then there is $n \in \mathbb{Z}$ such that $y_1 = y_2 + n$, and so
$y_1 - y_2 = n$. Consequently, 
$n \in (-\frac 1 2, \frac 1 2) + (-\frac 1 2, \frac 1 2) = (-1, 1)$.
Therefore, $n \in \mathbb{Z} \cap (-1,1) = \{0\}$. Hence,
$y_1 = y_2$, as required.
\end{remark}

\begin{theorem} \label{thm:nec:pi-Y}
Let $Y\subseteq (-\frac 1 2, \frac 1 2)$ be a compact subset of 
$\mathbb{R}$. If the projection $\pi(Y)$ is quasi-convex in~$\mathbb{T}$, 
then $Y$ is quasi-convex in $\mathbb{R}$.
\end{theorem}

\begin{proof}
Since $\pi(Y)$ is quasi-convex in $\mathbb{T}$ and $\pi$ is a continuous 
group homomorphism, by Proposition~\ref{prop:intro:qc-hom}, 
\mbox{$Q_\mathbb{R}(Y) \subseteq \pi^{-1}(\pi(Y))$}.
On the other hand, there 
is $0<M < \frac 1 2$ such 
that $Y \subseteq [-M,M]$, because $Y$ is compact. Thus,
$Q_\mathbb{R}(Y) \subseteq [-M,M]$, as
$[-M,M]$ is quasi-convex (cf. \cite[p.~79]{BruPhD}). 
Therefore, by Remark~\ref{rem:nec:pi}, 
$\pi^{-1}(\pi(Y)) \cap (-\tfrac 1 2, \tfrac 1 2) = Y$, and so
\begin{align}
Q_\mathbb{R}(Y) \subseteq  \pi^{-1}(\pi(Y)) \cap [-M,M] \subseteq
\pi^{-1}(\pi(Y)) \cap (-\tfrac 1 2, \tfrac 1 2) = Y.
\end{align}
Hence, $Y$ is quasi-convex in $\mathbb{R}$, as desired.
\end{proof}

\begin{corollary} \label{cor:nec:pi-Y}
Let $Y\subseteq \mathbb{R}$ be a compact subset. If there is $\alpha \neq 0$ 
such that \mbox{$\alpha Y \subseteq (-\frac 1 2, \frac 1 2)$} and 
$\pi(\alpha Y)$ is quasi-convex in $\mathbb{T}$, then $Y$ is quasi-convex 
in $\mathbb{R}$.
\end{corollary}

\begin{proof}
By Theorem~\ref{thm:nec:pi-Y}, the set $\alpha Y$  is quasi-convex in 
$\mathbb{R}$. The map $f\colon \mathbb{R}\rightarrow \mathbb{R}$ defined by 
\mbox{$f(y)=\alpha y$}
is a topological automorphism of $\mathbb{R}$. 
Consequently, \mbox{$Y=f^{-1}(\alpha Y)$}
is quasi-convex in~$\mathbb{R}$, 
because quasi-convexity is preserved by topological isomorphisms.
\end{proof}

\begin{theorem} \label{thm:nec:gen}
Suppose that $X$ is quasi-convex in $\mathbb{T}$. Then:
\begin{myalphlist}

\item
$b_0 \geq 4$;

\item
$|\{n \in \mathbb{N} \mid q_n=2 \}| \leq 1$;

\item
if $q_n=2$ for some $n\in \mathbb{N}$, then $q_{n+1}> 4$.

\end{myalphlist}
\end{theorem}

\begin{proof}
First, we show that \mbox{$b_0 \geq 4$}. 
Assume that \mbox{$b_0=2$} or \mbox{$b_0=3$}. (The 
case \mbox{$b_0=1$} is not possible, because $b_0 > 1$.) Then
$\langle \frac 1 {b_0} \rangle = \{0,\pm \frac 1 {b_0} \} \subseteq X$, 
and thus (using the fact that $H^\triangleright$ coincides with the 
annihilator $H^\perp$ for every subgroup $H$) one obtains that
\begin{align}
X^\triangleright \subseteq \langle \frac 1 {b_0} \rangle^\triangleright = 
\langle \frac 1 {b_0} \rangle^\perp= b_0 \mathbb{Z}. 
\end{align}
Consequently,
$\frac 1 {b_0} + X \subseteq X^{\triangleright\triangleleft} = 
Q_{\mathbb{T}}(X)$. Since the only accumulation point of $X$ is $0$, and 
the only accumulation point of $\frac 1 {b_0} + X$ is
$\frac 1 {b_0}$, it follows that $Q_{\mathbb{T}}(X) \not\subseteq X$, 
contrary to our assumption that $X$ is quasi-convex. This contradiction 
shows that $b_0 \geq 4$, and hence (a) holds.

Since $b_0 \geq 4$, the set $S$ satisfies that 
$S \subseteq [-\frac 1 4, \frac 1 4] \subseteq (-\frac 1 2, \frac 1 2)$, 
and $S$ is compact (because it is consists of zero and a sequence that 
converges to $0$). Clearly, $\pi(S)=X$, and by our assumption, it is 
quasi-convex.  So, by Theorem~\ref{thm:nec:pi-Y}, $S$ is quasi-convex in 
$\mathbb{R}$.  Thus, by Theorem~\ref{thm:nec:specR},
(b) and (c) hold, as desired.
\end{proof}

\begin{theorem} \label{thm:nec:div3}
If $X$ is quasi-convex in $\mathbb{T}$, then for every $n\in \mathbb{N}$,
\mbox{$4 \hspace{-2pt} \not\hspace{2.75pt}\mid   b_{n+1}$} implies that 
$q_n \neq 3$. In particular, if 
\mbox{$4 \hspace{-2pt} \not\hspace{2.75pt}\mid b_{n}$} 
for every $n \in\mathbb{N}$, then $q_n \neq 3$ for every 
$n \in\mathbb{N}$.
\end{theorem}

In order to prove Theorem~\ref{thm:nec:div3}, we rely on the following 
proposition, which holds in every abelian topological group. For
\mbox{$x\in G$}, put 
\mbox{$\operatorname{Tr}_x(G)=\{\chi(x) \mid \chi \in \widehat G\}$}.  
Clearly,  $\operatorname{Tr}_x(G)$ is a  subgroup of $ \mathbb{T}$, and 
\mbox{$\operatorname{Tr}_{2x}(G)=2\operatorname{Tr}_{x}(G)$}.

\begin{samepage}
\begin{proposition} \label{prop:nec:2x}
Let $G$ be an abelian topological group, and \mbox{$x \in G$}. Then the 
following are equivalent: 

\begin{myromanlist}

\item  
\mbox{$2x \in Q_{G}(\{x,3x\})$};

\item 
\mbox{$\pm  \frac 1 4\not \in \operatorname{Tr}_{x}(G)$} 
(i.e., \mbox{$\chi(x) \neq \pm \tfrac 1 4$} for every 
\mbox{$\chi \in \widehat G$}); 

\item 
\mbox{$\tfrac 1 2 \not \in \operatorname{Tr}_{2x}(G)$};

\item 
the subgroup  $\operatorname{Tr}_{2x}(G)$   of  $ \mathbb{T}$ has no 
non-zero 2-torsion elements.  

\end{myromanlist}
\end{proposition}
\end{samepage}

\begin{proof} (i) $\Rightarrow$ (ii): 
Assume that $\pm \tfrac 1 4 \in \operatorname{Tr}_x(G)$. Then there is
$\chi \in \widehat{G}$ such that $\chi(x)=\tfrac 1 4$. Thus,
$\chi(x)\in \mathbb{T}_+$, and $\chi(3x)=\tfrac 3 4 \in \mathbb{T}_+$, but
$\chi(2x)=\tfrac 1 2 \not \in \mathbb{T}_+$. 
Therefore, $2x \not\in Q_G(\{x,3x\})$, contrary to (i). This contradiction 
shows that $\pm \tfrac 1 4 \not\in \operatorname{Tr}_x(G)$.

(ii) $\Rightarrow$ (iii): If 
\mbox{$\frac 1 2 \in \operatorname{Tr}_{2x}(G)=2\operatorname{Tr}_x(G)$}, then
there is \mbox{$\chi \in \widehat{G}$} such that 
\mbox{$2\chi(x) =\tfrac 1 2$}, and thus \mbox{$\chi(x)=\pm \tfrac 1 4$}, 
contrary to (ii). Hence, \mbox{$\frac 1 2 \not\in\operatorname{Tr}_{2x}(G)$}.

(iii) $\Rightarrow$ (iv): Since the Pr\"ufer group $\mathbb{Z}(2^\infty)$ 
is the $2$-torsion subgroup of $\mathbb{T}$, if 
$\operatorname{Tr}_{2x}(G)$ contains a non-zero $2$-torsion element $t$ 
of order $2^l$, then  \mbox{$2^{l-1}t = \tfrac 1 2 \in 
\operatorname{Tr}_{2x}(G)$}, contrary to (iii).
Hence, $\operatorname{Tr}_{2x}(G) \cap \mathbb{Z}(2^\infty)=\{0\}$.

(iv) $\Rightarrow$ (ii): Since 
$\operatorname{Tr}_{2x}(G)=2\operatorname{Tr}_x(G)$, this implication is 
obvious.

(ii) $\Rightarrow$ (i):
Suppose that 
\mbox{$\chi \hspace{-1pt}\in \hspace{-1pt}\{x,3x\}^\triangleright$}. 
Then \mbox{$\chi(x) \hspace{-1pt}\in \mathbb{T}_+$} and
\mbox{$3\chi(x)\hspace{-1pt}=\hspace{-1pt}\chi(3x)
\hspace{-1pt}\in \hspace{-1pt}  \mathbb{T}_+$}.  
Consequently,
\mbox{$\chi(x) \in \{\pm \frac 1 4\} \cup \mathbb{T}_3$}. By (b),
this implies that 
\mbox{$\chi(x) \in \mathbb{T}_3 \subseteq \mathbb{T}_2$.}
Therefore, \mbox{$2\chi(x) = \chi(2x) \in \mathbb{T}_+$.}
Hence, \mbox{$2x \in Q_{G}(\{x,3x\})$}, as desired.
\end{proof}

\begin{samepage}
\begin{corollary} \label{cor:nec:2x}
Let $G$ be an abelian topological group, and $x \in G$. If
\begin{myalphlist}

\item the order of $2x$ is finite and odd; or

\item 
$ \widehat G$ is torsion, and the order of every 
\mbox{$\chi \in \widehat G$} such  that $\chi(x)\neq 0$  is not divisible 
by $4$; 
 
\vspace{6pt}
\end{myalphlist}
then $2x \in Q_{G}(\{x,3x\})$.
\end{corollary}
\end{samepage}

\begin{proof}
(a)  Since each $\chi \in \widehat G$ is a homomorphism, 
the order of $\chi(2x)$ divides the order of $2x$. So,
if $\chi(2x)=\frac 1 2$, then $2$ divides the order of $2x$, contrary 
to our assumption. Therefore, $\chi(2x)\neq \pm \frac 1 2$ for
all $\chi \in \widehat G$, and the statement follows by 
Proposition~\ref{prop:nec:2x}(iii).

(b) The map $\hat x \colon \widehat G \rightarrow \mathbb{T}$ defined 
by $\hat x(\chi)=\chi(x)$ is a homomorphism, and so the order of 
$\chi(x)$ divides the order of $\chi$. Thus, if $\chi(x)=\pm \frac 1 4$, 
then $4$ divides the order of $\chi$, contrary
to our assumption. Hence, \mbox{$\chi(x)\neq \pm \frac 1 4$} for
all \mbox{$\chi \in \widehat G$}, and the statement follows by 
Proposition~\ref{prop:nec:2x}(ii).
\end{proof}

\begin{proof}[Proof of Theorem~\ref{thm:nec:div3}.]
Suppose that $q_n=3$, and put $x=\frac{1}{b_{n+1}}$. Then
$3x=\frac{1}{b_n}$, and thus $\{x,3x\} \subseteq X$. 
Since $b_{n+1}$ is not divisible by $4$, the order of $2x$ is odd, and so 
by Corollary~\ref{cor:nec:2x}(a), 
$2x \in Q_{\mathbb{T}}(\{x,3x\})$. Consequently,
\begin{align} \label{eq:nec:div3}
2x \in Q_{\mathbb{T}}(\{x,3x\}) \subseteq Q_{\mathbb{T}}(X) = X, 
\end{align}
because $X$ is quasi-convex.  By  Theorem~\ref{thm:nec:gen}, 
$b_1 > b_0 \geq 4$, and so $X=\pi(S)$, where 
$S\subseteq [-\frac 1 4, \frac 1 4]$. Therefore, by Remark~\ref{rem:nec:pi},
(\ref{eq:nec:div3}) implies that $\frac{2}{b_{n+1}} \in S$. This, 
however, is impossible, because 
$\frac{1}{b_{n+1}} < \frac{2}{b_{n+1}} < \frac {1}{b_n}$. This 
contradiction shows that $q_n \neq 3$, as desired.
\end{proof}

Corollary~\ref{cor:nec:2x} also has a useful application for the group 
$\mathbb{J}_p$ of $p$-adic integers.

\begin{corollary}\label{cor:nec:J-2x}
Let $p>2$ be a prime, and
$x \in \mathbb{J}_p$. Then $2x \in Q_{\mathbb{J}_p}(\{x,3x\})$.
\end{corollary}

\begin{proof}
Recall that $\widehat{\mathbb{J}_p} = \mathbb{Z}(p^\infty)$, and so every 
element in  $\widehat{\mathbb{J}_p}$ is of order $p^k$ for some 
\mbox{$k\in\mathbb{N}$}. Thus, $4$ does not divide the order of any 
element in $\widehat{\mathbb{J}_p}$. Therefore,  the statement follows by 
Corollary~\ref{cor:nec:2x}(b).
\end{proof}


\section{Sequences of the form $\boldsymbol{2^{-(a_n+1)}}$ in 
$\boldsymbol{\mathbb{T}}$ and $\boldsymbol{\mathbb{R}}$}

\label{sect:T2R2}

In this section, we prove Theorem~\ref{thm:T2:main} and~\ref{thm:R2:main}. 
The case of Theorem~\ref{thm:T2:main} where $g_n > 1$ for every 
$n\in\mathbb{N}$ follows from Theorem~\ref{thm:intro:DikLeo}(a). Thus, the 
main thrust of the argument in the proof of Theorem~\ref{thm:T2:main} 
is~to show that the conclusion remains true if one permits \mbox{$g_{n_0}=1$} 
for a single index \mbox{$n_0\in \mathbb{N}$}. This is carried out by 
induction on $n_0$. We start off with proving 
sufficiency of the conditions of Theorem~\ref{thm:T2:main} in a special 
case, where $n_0=0$ and $a_0=1$.

\begin{proposition} \label{prop:T2:spec}
Suppose that \mbox{$a_0=1$}, \mbox{$a_1=2$}, \mbox{$2< g_1$}, and 
\mbox{$1< g_n$} for all \mbox{$n \geq 2$}. 
Then $K_{\underline a,2}$ is quasi-convex in $\mathbb{T}$.
\end{proposition}

\begin{notation}
For $g \in \mathbb{Z}$, we denote by $\underline{a+g}$ the sequence whose 
$n$th term is $a_n+g$.
\end{notation}

\begin{proof}[Proof of Proposition~\ref{prop:T2:spec}.]
Let \mbox{$\underline{a}^\prime =\{a_n^\prime\}_{n=0}^\infty$}
denote the sequence defined by \mbox{$a_n^\prime = a_{n+2}$}. Then
for every $n \in \mathbb{N}$,
\begin{align} \label{eq:T2:a'}
a^\prime_{n} & =a_{n+2} \geq a_2 = a_1+g_1 \geq 5, \text{ and} \\
\label{eq:T2:g'}
g_n^\prime &=a_{n+1}^\prime - a_n^\prime = a_{n+3}-a_{n+2} = g_{n+2} >1.
\end{align}
Observe that
\begin{align}\label{eq:T2:Kaa'}
K_{\underline{a},2} = \{\pm \tfrac 1 4,\pm \tfrac 1 8\}\cup
K_{\underline{a}^\prime,2}.
\end{align}
Let $f\colon \hspace{-0.15pt}\mathbb{T} \rightarrow \mathbb{T}$ denote the 
continuous homomorphism
defined by \mbox{$f(x)=8x$}. By (\ref{eq:T2:Kaa'}),
\begin{align}
f(K_{\underline{a},2}) = 8 K_{\underline{a^\prime},2} = 
K_{\underline{a^\prime-3},2}. 
\end{align}
Consequently, by Proposition~\ref{prop:intro:qc-hom},
\begin{align} \label{eq:T2:QK}
Q_\mathbb{T}(K_{\underline{a},2}) \subseteq 
f^{-1}(Q_\mathbb{T}(K_{\underline{a^\prime-3},2})).
\end{align}
By (\ref{eq:T2:a'}), one has $a_0^\prime -3 \geq 2$, and by (\ref{eq:T2:g'}),
$(a_{n+1}^\prime -3) - (a_n^\prime-3)=g_n^\prime > 1$ for 
every $n\in \mathbb{N}$.  Thus, according to
Theorem~\ref{thm:intro:DikLeo}(a),
$K_{\underline{a^\prime-3},2}$ is quasi-convex in $\mathbb{T}$, and so by
(\ref{eq:T2:QK}),
\begin{align}
Q_\mathbb{T}(K_{\underline{a},2}) \subseteq 
f^{-1}(K_{\underline{a^\prime-3},2}) = 
\bigcup\limits_{i=-3}^4 (\tfrac i 8 + K_{\underline{a}^\prime,2}).
\end{align}
Since $a_0^\prime \geq 5$ by (\ref{eq:T2:a'}), one has
$K_{\underline{a}^\prime,2} \subseteq \mathbb{T}_{16}$. So, for
$i=-3,3,4$, 
\begin{align} \label{eq:T2:-334}
(\tfrac{i}{8} + K_{\underline{a}^\prime,2}) \cap \mathbb{T}_+ 
\subseteq (\tfrac{i}{8} + \mathbb{T}_{16}) \cap \mathbb{T}_+ = \emptyset.
\end{align}
As $K_{\underline{a},2} \subseteq \mathbb{T}_+$ and 
$\mathbb{T}_+ = \{1\}^\triangleleft$ is quasi-convex, 
$Q_\mathbb{T}(K_{\underline{a},2}) \subseteq \mathbb{T}_+$.
Therefore, by (\ref{eq:T2:-334}),
\begin{align} \label{eq:T2:QK2}
Q_\mathbb{T}(K_{\underline{a},2}) \subseteq
\bigcup\limits_{i=-2}^2 (\tfrac i 8 + K_{\underline{a}^\prime,2}).
\end{align}
Recall that $K_{\underline{a}^\prime,2} \subseteq \mathbb{T}_{16}$. 
Thus,
\begin{align}
7 K_{\underline{a}^\prime,2} \subseteq 7 \mathbb{T}_{16}
\subseteq 8 \mathbb{T}_{16}
\subseteq \mathbb{T}_2 \subseteq \mathbb{T}_+,
\end{align}
and $7 \in K_{\underline{a}^\prime,2}^\triangleright$.
Since $7\in\{\pm \frac 1 4,\pm \frac 1 8\}^\triangleright$, it follows 
from (\ref{eq:T2:Kaa'}) that
\begin{align}
7 \in \{\pm \tfrac 1 4,\pm \tfrac 1 8\}^\triangleright 
\cap  K_{\underline{a}^\prime,2}^\triangleright = 
(\{\pm \tfrac 1 4,\pm \tfrac 1 8\}\cup
K_{\underline{a}^\prime,2})^\triangleright = 
K_{\underline{a},2}^\triangleright.
\end{align}
Consequently,
\begin{align}
Q_\mathbb{T}(K_{\underline{a},2}) =
K_{\underline{a},2}^{\triangleright\triangleleft} \subseteq
\{1,7\}^\triangleleft = \{\pm \tfrac 1 4\} \cup
\mathbb{T}_{7}\cup (\tfrac 1 7 + \mathbb{T}_7) \cup
(- \tfrac 1 7 + \mathbb{T}_7).
\end{align}
Therefore,
\begin{align}
Q_\mathbb{T}(K_{\underline{a},2}) \cap
(\tfrac 1 4 + K_{\underline{a}^\prime,2}) \subseteq
Q_\mathbb{T}(K_{\underline{a},2}) \cap
(\tfrac 1 4 + \mathbb{T}_{16})\subseteq
(\tfrac 1 7 + \mathbb{T}_{7}) \cap
(\tfrac 1 4 + \mathbb{T}_{16}) = \{\tfrac 1 4\}.
\end{align}
Similarly, 
\begin{align}
Q_\mathbb{T}(K_{\underline{a},2}) \cap 
(- \tfrac 1 4 + K_{\underline{a}^\prime,2}) \subseteq \{- \tfrac 1 4\}.
\end{align}
Hence, by (\ref{eq:T2:QK2}),
\begin{align} \label{eq:T2:QK3}
Q_\mathbb{T}(K_{\underline{a},2}) \subseteq \{\pm \tfrac 1 4\} \cup
K_{\underline{a}^\prime,2} \cup (\tfrac 1 8 + K_{\underline{a}^\prime,2})
\cup (-\tfrac 1 8 + K_{\underline{a}^\prime,2}).
\end{align}
In light of (\ref{eq:T2:Kaa'}) and (\ref{eq:T2:QK3}), we conclude by 
showing that
\begin{align}
Q_\mathbb{T}(K_{\underline{a},2}) \cap 
(\tfrac 1 8 + K_{\underline{a}^\prime,2})\subseteq \{\tfrac 1 8\}.
\end{align}
It will follow by symmetry that
\begin{align}
Q_\mathbb{T}(K_{\underline{a},2}) \cap
(-\tfrac 1 8 + K_{\underline{a}^\prime,2})\subseteq \{-\tfrac 1 8\}.
\end{align}
To that end, we prove that for every 
$x \in \tfrac 1 8 + K_{\underline{a}^\prime,2}$, if $x \neq \tfrac 1 8$,
then there is $k \in K_{\underline{a},2}^\triangleright$ such that
$kx \not \in \mathbb{T}_+$.
Let $x=\frac 1 8 \pm x_n \in (\tfrac 1 8 + K_{\underline{a}^\prime,2})$.
Then $n \geq 2$, so $x_n \in \mathbb{T}_{16}$, and
\begin{align}
 (2^{a_n-1}-1)(\tfrac 1 8 - x_n) & = 
(2^{a_n-1}-1)\tfrac 1 8 - (2^{a_n-1}-1) x_n   \\
& \equiv_1 -\tfrac 1 8 -\tfrac 1 4 + x_n = 
-\tfrac 3 8 + x_n \not\in \mathbb{T}_+,\\
 (2^{a_n-1}-7)(\tfrac 1 8 + x_n) & = (2^{a_n-1}-7)\tfrac 1 8 +
(2^{a_n-1}-7) x_n   \\
& \equiv_1 -\tfrac 7 8 +\tfrac 1 4 -7 x_n 
 \equiv_1 \tfrac 3 8 - 7x_n \not\in \mathbb{T}_+.
\end{align}
It remains to show that 
\mbox{$2^{a_n-1}-1$} and  \mbox{$2^{a_n-1}-7$} are indeed in 
$K_{\underline{a},2}^\triangleright$. By (\ref{eq:T2:a'}),
\mbox{$a_n \geq 5$} for every \mbox{$n \geq 2$}.
Thus, $2^{a_n-1}-1, 2^{a_n-1}-7 \in 
\{\pm \tfrac 1 4, \pm \tfrac 1 8\}^\triangleright$.
If $2 \leq m < n$, then
$2^{a_n-1} x_m = 2^{a_n-a_m-2} \equiv_1 0$ (because $g_n >1$ for every $n 
\geq 2$). Consequently, 
\begin{align}
(2^{a_n-1}-1)x_m = -x_m \in \mathbb{T}_+ \quad
\text{and} \quad (2^{a_n-1}-7)x_m 
= -7x_m \in 7\mathbb{T}_{16} \subseteq \mathbb{T}_+. 
\end{align}
On the other hand, if
$n \leq m$, then $x_m \in \mathbb{T}_{2^{a_m-1}}$. Therefore,
\begin{align}
(2^{a_n-1}-1)x_m \subseteq (2^{a_n-1}-1) \mathbb{T}_{2^{a_m-1}} \subseteq
2^{a_n-1} \mathbb{T}_{2^{a_m-1}} \subseteq 
2^{a_m-1} \mathbb{T}_{2^{a_m-1}} = \mathbb{T}_+, \\
(2^{a_n-1}-7)x_m \subseteq (2^{a_n-1}-7) \mathbb{T}_{2^{a_m-1}} \subseteq
2^{a_n-1} \mathbb{T}_{2^{a_m-1}} \subseteq
2^{a_m-1} \mathbb{T}_{2^{a_m-1}} = \mathbb{T}_+.
\end{align}
Hence, by (\ref{eq:T2:Kaa'}),
\begin{align}
2^{a_n-1}-1, 2^{a_n-1}-7 \in 
\{\pm \tfrac 1 4, \pm \tfrac 1 8\}^\triangleright \cap
K_{\underline{a}^\prime,2}^\triangleright = 
(\{\pm \tfrac 1 4,\pm \tfrac 1 8\}\cup
K_{\underline{a}^\prime,2})^\triangleright = 
K_{\underline{a},2}^\triangleright,
\end{align}
as desired.
\end{proof}

\begin{lemma} \label{lemma:T2:div}
Let $Y \subseteq \mathbb{T}$ be symmetric, $m\in 
\mathbb{N}\backslash\{0\}$, and $0 < k \leq 2m$.

\begin{myalphlist}

\item
If $Y \subseteq \mathbb{T}_m$ and $kY$ is quasi-convex in
$\mathbb{T}$, then $Y$ is quasi-convex too.

\item
If $Y \subseteq \mathbb{T}_{4m}$ and $4mY$ is quasi-convex in $\mathbb{T}$, 
then $Y^\prime=\{\pm \frac 1 {4m}\} \cup Y$ is quasi-convex too.

\end{myalphlist}
\end{lemma}

\begin{proof}
For $x\in \mathbb{T}$, let $\|x\|$ denote the distance of $x$ from $0$, 
that is, with some abuse of notations,
$\|x\|=\inf\limits_{n \in \mathbb{Z}} |x+n|$.

(a) Let $f\colon \mathbb{T} \rightarrow  \mathbb{T}$ denote the continuous 
homomorphism defined by $f(x)=kx$. Since $kY$ is quasi-convex in $\mathbb{T}$,
by 
Proposition~\ref{prop:intro:qc-hom},
\begin{align}
Q_\mathbb{T}(Y) \subseteq f^{-1}(Q_\mathbb{T}(f(Y))) =
f^{-1}(Q_\mathbb{T}(kY)) = f^{-1}(kY) = 
\bigcup\limits_{i=0}^{k-1} (\tfrac i k + Y).
\end{align}
On the other hand, $Q_\mathbb{T}(Y) \subseteq \mathbb{T}_m$, because 
$Y \subseteq \mathbb{T}_m$, and 
$\mathbb{T}_m=\{1,\ldots, m\}^\triangleleft$ 
is quasi-convex. Thus,
\begin{align}
Q_\mathbb{T}(Y) \subseteq
\bigcup\limits_{i=0}^{k-1} (\tfrac i k + Y)\cap\mathbb{T}_m.
\end{align}
If $x= \frac i k + y \in \frac i k + Y$, where $y \in Y$ and $i\neq 0$, then 
\begin{align}
\tfrac 1 {4m} = \tfrac 1 {2m} - \tfrac 1 {4m}
\leq \tfrac 1 k - \tfrac 1 {4m} \leq \| \tfrac i k \| - \|y\| \leq \|x\|,
\end{align}
with equality if and only if $k=2m$, and 
\begin{myromanlist}
\item
$y= - \frac 1 {4m}$ and $i=1$; or

\item
$y=\frac 1 {4m}$ and $i=k-1$.

\vspace{10pt}
\end{myromanlist}
In either case, $x=-y \in Y$, because $Y$ is symmetric. In all other 
cases (e.g., when $k\neq 2m$), one has $x \not \in \mathbb{T}_m$. 
Therefore,
\begin{align}
Q_\mathbb{T}(Y) \subseteq
\bigcup\limits_{i=0}^{k-1} (\tfrac i k + Y)\cap \mathbb{T}_m = Y, 
\end{align}
as desired.

(b)  First, suppose that $m=1$. Then $Y \subseteq \mathbb{T}_4$, and so
$Y^\prime \subseteq \{\pm \frac 1 4 \} \cup \mathbb{T}_4$. 
The set $\{\pm \frac 1 4 \} \cup \mathbb{T}_4$ is quasi-convex, because it 
is equal to $\{1,3,4\}^\triangleleft$. Thus, 
\begin{align} \label{eq:T2:QY}
Q_\mathbb{T}(Y^\prime) \subseteq \{\pm \tfrac 1 4 \} \cup \mathbb{T}_4.
\end{align}
Let $f\colon \mathbb{T} \rightarrow  \mathbb{T}$ denote the continuous 
homomorphism defined by $f(x)=4x$. Since  $4Y^\prime = 4Y$ is 
quasi-convex in $\mathbb{T}$, by Proposition~\ref{prop:intro:qc-hom},
\begin{align}
Q_\mathbb{T}(Y^\prime) \subseteq f^{-1}(Q_\mathbb{T}(f(Y^\prime)))  =
f^{-1}(Q_\mathbb{T}(4Y^\prime)) & = f^{-1}(4Y^\prime) \\
& = Y^\prime \cup (\tfrac 1 4 + Y^\prime) \cup 
(\tfrac 1 2 + Y^\prime) \cup (- \tfrac 1 4 + Y^\prime).
\end{align}
As $Y^\prime \subseteq \mathbb{T}_+$, 
\begin{align}
(\tfrac 1 2 + Y^\prime) \cap (\{\pm \tfrac 1 4 \} \cup \mathbb{T}_4) = 
\{\pm \tfrac 1 4 \} \subseteq Y^\prime. 
\end{align}
Since
\begin{align}
\tfrac 1 4 + Y^\prime = (\tfrac 1 4 + Y) \cup \{0,\tfrac 1 2\} \subseteq
(\tfrac 1 4 + \mathbb{T}_4) \cup \{0,\tfrac 1 2\},
\end{align}
one has
\begin{align}
(\tfrac 1 4 + Y^\prime) \cap (\{\pm \tfrac 1 4 \} \cup \mathbb{T}_4) =
\{0,\tfrac 1 4\} \subseteq Y^\prime. 
\end{align}
Similarly,
\begin{align}
(-\tfrac 1 4 + Y^\prime) \cap (\{\pm \tfrac 1 4 \} \cup \mathbb{T}_4) =
\{0,\tfrac 1 4\} \subseteq Y^\prime.
\end{align}
Hence, by (\ref{eq:T2:QY}),
$Q_{\mathbb{T}}(Y^\prime) \subseteq Y^\prime$, as desired. 

If \mbox{$m>1$}, then put \mbox{$Z=mY$} and 
\mbox{$Z^\prime = m Y^\prime = \{ \pm \frac 1 4 \} \cup Z$}. 
Since \mbox{$Y \subseteq \mathbb{T}_{4m}$},
one has \mbox{$Z \subseteq \mathbb{T}_4$}. By our assumption, 
\mbox{$4Z = 4m Y$} is quasi-convex in $\mathbb{T}$. Thus, by what we have 
shown so far,
\mbox{$Z^\prime=mY^\prime$} is quasi-convex. By (a), this implies that 
$Y^\prime$ 
is quasi-convex, because
\mbox{$Y^\prime \subseteq \mathbb{T}_{m}$}  and \mbox{$m \leq 2m$}.
\end{proof}

Using Lemma~\ref{lemma:T2:div}, the assumption that $a_0=1$ in 
Proposition~\ref{prop:T2:spec} can be sufficiently relaxed to cover the 
base of the inductive proof of Theorem~\ref{thm:T2:main}.

\begin{corollary} \label{cor:T2:spec}
Suppose that $0<a_0$, $g_0=1$, $2<g_1$, and $1< g_n$ for all $n \geq 2$.
Then $K_{\underline a,2}$ is quasi-convex.
\end{corollary}

\begin{proof}
Since the sequence $\underline{a}$ is increasing, one has
$K_{\underline a,2} \subseteq \mathbb{T}_{2^{a_0-1}}$.
The sequence $\underline{a-a_0+1}$ satisfies the conditions of 
Proposition~\ref{prop:T2:spec}, and thus
$2^{a_0-1} K_{\underline{a},2}=K_{\underline{a-a_0+1},2}$
is quasi-convex. Hence, by Lemma~\ref{lemma:T2:div}(a) 
with $k=m=2^{a_0-1}$, one obtains that $K_{\underline{a},2}$ is 
quasi-convex, as desired.
\end{proof}

\begin{proof}[Proof of Theorem~\ref{thm:T2:main}.]
Suppose that $K_{\underline{a}, 2}$ is quasi-convex in $\mathbb{T}$, and 
put  \mbox{$b_n=2^{a_n+1}$.}
Using the notations of Section~\ref{sect:nec},
\begin{align}
q_n = \dfrac{b_{n+1}}{b_n} = 2^{a_{n+1}-a_n} = 2^{g_n}. 
\end{align}
By Theorem~\ref{thm:nec:gen}(a), $b_0=2^{a_0+1} \geq 4$, and thus
$a_0 \geq 1$.
By Theorem~\ref{thm:nec:gen}(b), there is at most one index
\mbox{$n \in \mathbb{N}$}
such that $q_n=2$. Therefore, $g_n=1$ for at most one index $n$,
and so $g_n >1$ for all but one $n \in \mathbb{N}$, because 
$\{a_n\}_{n=0}^\infty$ is
increasing. By 
Theorem~\ref{thm:nec:gen}(c), if $q_n=2$ for some $n \in \mathbb{N}$, 
then $q_{n+1} > 4$. Hence, if $g_n=1$ for some $n \in \mathbb{N}$, 
then $g_{n+1}>2$.

Conversely, suppose that $\underline{a}$ satisfies (i)--(iii). If
$g_n > 1$ for all $n \in \mathbb{N}$, then by 
Theorem~\ref{thm:intro:DikLeo}(a), $K_{\underline a, 2}$ is quasi-convex
in $\mathbb{T}$, and there is nothing left to prove. So, we may assume 
that there is $n_0 \in \mathbb{N}$ such that $g_{n_0}=1$. In this case, 
by (ii), there is precisely one such index $n_0$. We proceed by induction 
on $n_0$.

If $n_0=0$, then by Corollary~\ref{cor:T2:spec}, $K_{\underline a, 2}$ is 
quasi-convex in $\mathbb{T}$. For the inductive step, suppose that the 
theorem holds for  all sequences such that $g_{n_0}=1$. Let $\underline a$ 
be a sequence that satisfies (i)--(iii) such that $g_{n_0 +1}=1$.
Let $\underline{a}^\prime = \{a_n^\prime\}_{n=0}^\infty$ denote the 
sequence 
defined by $a_n^\prime = a_{n+1}-a_0 -1$. As $g_{n_0+1}=1$, by (ii),
$g_n > 1$ for all $n \neq n_0+1$. In particular, $g_0 > 1$, and so
\begin{align}
a_0^\prime = a_1 -a_0-1 =g_0-1 >0.
\end{align}
This shows that $\underline{a}^\prime$ satisfies (i).
Since $\underline a$ satisfies (ii)--(iii) and
\begin{align}
g_n^\prime = a_{n+1}^\prime -a_n^\prime =
(a_{n+2} -a_0 -1) - (a_{n+1} -a_0 -1) = a_{n+2}- a_{n+1}=g_{n+1},
\end{align}
so does $\underline{a}^\prime$, and $g_{n_0}^\prime=1$. Thus, by the 
inductive hypothesis, $K_{\underline{a}^\prime,2}$ is quasi-convex in 
$\mathbb{T}$. Put \mbox{$m=2^{a_0-1}$}. Clearly,
\begin{align}
4m K_{\underline{a^\prime +a_0+1},2} =
4\cdot 2^{a_0-1} K_{\underline{a^\prime +a_0+1},2} = 
2^{a_0+1} K_{\underline{a^\prime +a_0+1},2} =  
K_{\underline{a^\prime},2}.
\end{align}
The smallest member of the sequence  $\underline{a^\prime +a_0+1}$ is 
\begin{align}
a_0^\prime+a_0+1 = a_1=a_0+g_0 \geq a_0+2.
\end{align}
Therefore,
\begin{align}
K_{\underline{a^\prime +a_0+1},2} \subseteq
\mathbb{T}_{2^{(a_0^\prime+a_0+1)-1}} = 
\mathbb{T}_{2^{a_0+g_0-1}} \subseteq \mathbb{T}_{2^{a_0+1}} = 
\mathbb{T}_{4m}.
\end{align}
Hence, by Lemma~\ref{lemma:T2:div}(b) applied
to the set $Y=K_{\underline{a^\prime +a_0+1},2}$,
\begin{align}
\{\pm\tfrac{1}{4m}\} \cup K_{\underline{a^\prime +a_0+1},2} = 
\{\pm\tfrac{1}{2^{a_0+1}}\} \cup K_{\underline{a^\prime +a_0+1},2} = 
\{\pm 2^{-(a_0+1)}\} \cup K_{\underline{a^\prime +a_0+1},2} = 
K_{\underline{a},2}
\end{align}
is quasi-convex in $\mathbb{T}$. This completes the proof. 
\end{proof}

\begin{proof}[Proof of Theorem~\ref{thm:R2:main}.]
Suppose that $R_{\underline{a}, 2}$ is quasi-convex in $\mathbb{R}$, and 
put \mbox{$b_n=2^{a_n+1}$}.  
Using the notations of Section~\ref{sect:nec},
\begin{align}
q_n = \dfrac{b_{n+1}}{b_n} = 2^{a_{n+1}-a_n} = 2^{g_n}. 
\end{align}
By Theorem~\ref{thm:nec:specR}(a), there is at most one index
\mbox{$n \in \mathbb{N}$}
such that $q_n=2$. Thus, $g_n=1$ for at most one index $n$,
and so $g_n >1$ for all but one $n \in \mathbb{N}$, because 
$\{a_n\}_{n=0}^\infty$ is
increasing. By 
Theorem~\ref{thm:nec:specR}(b), if $q_n=2$ for some $n \in \mathbb{N}$, 
then $q_{n+1} > 4$. Consequently, if $g_n=1$ for some $n \in \mathbb{N}$, 
then $g_{n+1}>2$.

Suppose that $g_n > 1$ for all but possibly one index $n \in \mathbb{N}$,
and if $g_n =1$ for some $n \in \mathbb{N}$, then $g_{n+1}>2$. Then 
$a_n^\prime=a_n-a_0+1$ satisfies the same conditions, because
$g_n^\prime=a_{n+1}^\prime -a_n^\prime =g_n$, and furthermore,
$a_0^\prime =1 >0$. Thus, by Theorem~\ref{thm:T2:main},
$\pi(R_{\underline{a}^\prime,2})=K_{\underline{a}^\prime,2}$ 
is quasi-convex in $\mathbb{T}$. We note that
for $\alpha = 2^{1-a_0}$, one has 
$\alpha R_{\underline{a},2} = R_{\underline{a}^\prime,2} 
\subseteq [-\frac 1 4, \frac 1 4]$. Therefore, by 
Corollary~\ref{cor:nec:pi-Y}, $R_{\underline{a},2}$ is quasi-convex in 
$\mathbb{R}$.
\end{proof}


\section{Sequences of the form $\boldsymbol{3^{-(a_n+1)}}$ in
$\boldsymbol{\mathbb{T}}$}

\label{sect:T3}

In this section, we present the proof of Theorem~\ref{thm:T3:main}.
Recall  that the Pontryagin dual $\widehat{\mathbb{T}}$ of 
$\mathbb{T}$ is $\mathbb{Z}$. For $k \in \mathbb{N}$, let 
$\eta_k\colon \mathbb{T} \rightarrow \mathbb{T}$ denote the continuous 
character defined by $\eta_k(x)=3^{k}\cdot x$. For $m \in \mathbb{N}$,
put $J_m = \{ k \in \mathbb{N} \mid m\eta_k \in 
K_{\underline a,3}^\triangleright \}$ and
$Q_m = \{m\eta_ k \mid k \in J_m\}^\triangleleft$.

\begin{lemma} \label{lemma:T3:J12}
$J_1 = J_2 = \mathbb{N} \backslash \underline{a}$.
\end{lemma}

\begin{proof}
One has $m\eta_k(x_n) = m \cdot 3^{k-a_n-1}$. Clearly,
$3^{i} \in \mathbb{T}_+$ if and only if $i \neq -1$, and
$2\cdot 3^{i} \in \mathbb{T}_+$ if and only if $i \neq -1$.
Thus, for $m=1,2$, $m\eta_k(x_n) \in \mathbb{T}_+$ if and only if
$k-a_n-1 \neq -1$, or equivalently, $a_n \neq k$.
\end{proof}

\begin{theorem} \label{thm:T3:Q12}
If $g_n > 1$ for every $n \in \mathbb{N}$, then 
$Q_1 \cap Q_2 = \{
\sum\limits_{n=0}^\infty \varepsilon_n x_n \mid
(\forall n \in \mathbb{N}) (\varepsilon_n \in \{-1,0,1\})\}$.
\end{theorem}

In order to prove Theorem~\ref{thm:T3:Q12}, we rely on an auxiliary 
statement. We identify points of $\mathbb{T}$ with $(-1/2,1/2]$. Recall 
that every
$y \in (-1/2,1/2]$ can be written in the form
\begin{align} \label{eq:T3:repy}
y=\sum\limits_{i=1}^\infty \frac{c_i}{3^i} = 
\frac{c_1}{3} + \frac{c_2}{3^2} + \cdots + \frac{c_s}{3^s}+\cdots,
\end{align}
where $c_i\in\{-1,0,1\}$.

\begin{lemma} \label{lemma:T3:b1}
Let $y \in \mathbb{T}$. If $y \in \mathbb{T}_+$ and $2y \in \mathbb{T}_+$, 
then $c_1=0$ in every representation of $y$ in the form 
{\rm (\ref{eq:T3:repy})}.
\end{lemma}

\begin{proof}
Since $c_i \in \{-1,0,1\}$, one has
\begin{align}
\left|\sum\limits_{i=2}^\infty \dfrac{c_i}{3^i} \right| \leq
\sum\limits_{i=2}^\infty \dfrac{1}{3^i} = 
\dfrac{1}{9} \cdot \dfrac{3}{2} = \dfrac{1}{6}.
\end{align}
Thus,
\begin{align}
|y| \geq \left| \dfrac{c_1}{3} \right| - 
\left|\sum\limits_{i=2}^\infty \dfrac{c_i}{3^i} \right| \geq
\dfrac{|c_1|}{3} - \dfrac{1}{6}.
\end{align}
Since $y, 2y \in \mathbb{T}_+$, one has 
$y \in \mathbb{T}_2 =[-\frac 1 8, \frac 1 8]$. Therefore,
\begin{align}
\dfrac{|c_1|}{3} \leq |y| + \dfrac 1 6 \leq \dfrac 1 8 + \dfrac 1 6 = 
\dfrac{7}{24} < \dfrac 1 3.
\end{align}
Hence, $c_1=0$, as desired.
\end{proof}

\begin{proof}[Proof of Theorem~\ref{thm:T3:Q12}.] ($\subseteq$)
Let  $x=\sum\limits_{i=1}^\infty \dfrac{c_i}{3^i}$
be a representation of $x\in \mathbb{T}$ in the form (\ref{eq:T3:repy}). 
Then for every $k\in \mathbb{N}$, one has
\begin{align} \label{eq:T3:etak}
\eta_{k}(x) = 3^k  x = 
\sum\limits_{i=1}^\infty \dfrac{c_i}{3^{i-k}} \equiv_1 
\sum\limits_{i=k+1}^\infty \dfrac{c_i}{3^{i-k}} = 
\sum\limits_{i=1}^\infty \dfrac{c_{k+i}}{3^i}.
\end{align}
If $x \in Q_1 \cap Q_2$, then for every $k \in J_1=J_2$, the 
element $y=\eta_k(x)$ satisfies $y \in \mathbb{T}_+$ and
$2y \in \mathbb{T}_+$. Thus, by Lemma~\ref{lemma:T3:b1}, the coefficient 
of $\frac 1 3$ in (\ref{eq:T3:etak}) is zero, that is, $c_{k+1}=0$.
Since $J_1 = J_2 = \mathbb{N}\backslash \underline a$ holds by 
Lemma~\ref{lemma:T3:J12}, we have shown that
if $c_i\neq 0$, then $i-1 \not \in \mathbb{N}\backslash \underline a$, 
so $i -1 \in \underline a$, and therefore $i = a_n +1$ for some 
$n \in\mathbb{N}$. Hence, $x$ has the form
\begin{align}
x= \sum\limits_{n=0}^\infty \dfrac{c_{a_n+1}}{3^{a_n+1}} = 
\sum\limits_{n=0}^\infty \varepsilon_n x_n,
\end{align}
where $\varepsilon_n = c_{a_n+1} \in \{-1,0,1\}$.

($\supseteq$) Let $k \in J_1 = J_2 = \mathbb{N}\backslash \underline a$
(see Lemma~\ref{lemma:T3:J12}), 
and $x=\sum\limits_{n=0}^\infty \varepsilon_n x_n$. Then
$\eta_k(x) = \sum\limits_{n=0}^\infty \varepsilon_n \eta_k(x_n)$.
Let $n_0$ be the first index such that 
$\eta_{k}(z_{n_0}) \not \equiv_1 0$. Then $k < a_{n_0} +1$, and since
$k \not \in \underline a$, one has $k< a_{n_0}$. Consequently,
$\eta_{k}(z_{n_0}) \leq \frac 1 9$. Thus,
$\eta_{k}(z_{n_0+i}) \leq \frac{1}{9^{i+1}}$, because
$g_n >1$ implies that  $a_{n_0+i} - a_{n_0} \geq 2i$. Therefore,
\begin{align}
\left|\sum\limits_{n=n_0}^\infty \varepsilon_n \eta_k(x_n) \right| \leq
\sum\limits_{n=n_0}^\infty | \eta_k(x_n)| = 
\sum\limits_{i=0}^\infty | \eta_k(z_{{n_0}+i})| \leq
\sum\limits_{i=0}^\infty \dfrac 1 {9^{i+1}} = 
\dfrac 1 9 \cdot \dfrac{9}{8} = \dfrac{1}{8}.
\end{align}
So, 
\begin{align}
\eta_k(x) =
\sum\limits_{n=0}^\infty \varepsilon_n \eta_k(x_n)\equiv_1
\sum\limits_{n=n_0}^\infty \varepsilon_n \eta_k(x_n) \in \mathbb{T}_2.
\end{align}
Hence, $\eta_k(x) \in \mathbb{T}_+$ and  $2\eta_k(x) \in \mathbb{T}_+$, as 
desired.
\end{proof}

\begin{lemma} \label{lemma:T3:shift}
If $g_n > 1$ for all $n \in \mathbb{N}$, $g_0 > 0$, and
$0\leq k<l$ in $\mathbb{N}$, then $a_k-1 \in J_m$ for 
$m=3^{a_l-a_k}\pm 2$,  that is, 
$m \eta_{a_k-1}\in K_{\underline a,3}^\triangleright$.
\end{lemma}

\begin{proof}
Let $\chi=m\eta_{a_k-1}$ and $n\in \mathbb{N}$. If $n <k$, then
$a_n + 2 \leq a_k$ (as $g_n >1$), and so
\begin{align}
\chi(x_n) = m \eta_{a_k -1}(x_n) = m\cdot {3^{a_k -a_n-2}} \equiv_1 0.
\end{align}
If $k \leq n < l$, then $a_n+2 \leq a_l$ (since $g_n >1$), so
$a_n -a_k +2 \leq a_l -a_k$, and thus
\begin{align}
\chi(x_n) = m \cdot 3^{a_k-a_n-2} = 
\dfrac{m}{3^{a_n -a_k+2}} = 
\dfrac{3^{a_l-a_k}\pm 2}{3^{a_n -a_k+2}} \equiv_1
\pm \dfrac{2}{3^{a_n -a_k+2}} \in \mathbb{T}_+,
\end{align}
because $a_n-a_k+2 \geq 2$. Finally, if $l \leq n$, then
\begin{align}
\chi(x_n) = \dfrac{3^{a_l-a_k}\pm 2}{3^{a_n -a_k+2}} = 
\dfrac{1}{3^{a_n -a_l+2}}
\pm \dfrac{2}{3^{a_n -a_k+2}}.
\end{align}
Since $k<n$ and $g_n >1$, one has $a_k -a_n \geq 2$, and so
$\dfrac{2}{3^{a_n -a_k+2}} \in [0,\frac 2 {81}]$. Furthermore, $n \geq l$ 
implies that $a_n \geq a_l$, and so 
$\frac{1}{a_n -a_l+2} \in [0,\frac 1 9]$. Therefore,
$\chi(x_n)\in [0,\frac{11}{81}] \subseteq \mathbb{T}_+$.  This show that
$\chi=m\eta_{a_k-1} \in K^\triangleright_{\underline a, 3}$.
\end{proof}

\begin{remark}
Lemma~\ref{lemma:T3:shift} heavily depends on the assumption that 
$a_0 > 0$. Indeed, if $a_0=0$, then $\eta_{a_0-1}=\eta_{-1}$ is not
defined, and thus Lemma~\ref{lemma:T3:shift} makes no sense (and so fails) 
for $k=0$.
\end{remark}

\begin{proof}[Proof of Theorem~\ref{thm:T3:main}.]
Suppose that $K_{\underline{a}, 3}$ is quasi-convex, and 
put $b_n=3^{a_n+1}$.  Using the notations of Section~\ref{sect:nec},
\begin{align}
q_n = \dfrac{b_{n+1}}{b_n} = 3^{a_{n+1}-a_n} = 3^{g_n}. 
\end{align}
By Theorem~\ref{thm:nec:gen}(a), \mbox{$b_0 \geq 4$}, and thus \mbox{$a_0 >0$}.
Since $4$ does not divide $b_{n+1}$, by 
Theorem~\ref{thm:nec:div3}, $q_n \neq 3$. 
Consequently, $g_n \neq 1$.
Hence, $g_n > 1$ for every \mbox{$n\in\mathbb{N}$}, because 
$\{a_n\}_{n=0}^\infty$ is increasing.

Conversely, suppose that $a_0 >0$ and $g_n > 1$ for every $n\in 
\mathbb{N}$, and let 
$x \in Q_{\mathbb{T}}(K_{\underline{a},3})\backslash\{0\}$. 
We show that $x \in K_{\underline{a},3}$. By
Theorem~\ref{thm:T3:Q12}, $x$ can be expressed in the form
$x=\sum\limits_{n=0}^N \varepsilon_n x_n$, where
$N \in \mathbb{N}  \cup\{\infty\}$, and
$\varepsilon_{n_i} \in \{-1,1\}$ and $n_i < n_{i+1}$ for every $i$,
because $Q_{\mathbb{T}}(K_{\underline{a}}) \subseteq Q_1 \cap Q_2$.
Assume that $N >0$. By replacing $x$ with $-x$ if necessary, we may assume 
that $\varepsilon_{n_0}=1$. We put $\rho=\varepsilon_{n_1}$, so that
$x=x_{n_0}+\rho x_{n_1} + x^\prime$, where
\begin{align}
x^\prime = \begin{cases}
\sum\limits_{i=2}^N \varepsilon_{n_i} x_{n_i} & \text{if $N>1$} \\
0 & \text{if $N=1$}.
\end{cases}
\end{align}
Put \mbox{$k=n_0$}, \mbox{$l=n_1$}, 
\mbox{$m=3^{a_l-a_k}+ 2\rho$},
and \mbox{$\chi=m\eta_{a_k-1}$}.
By Lemma~\ref{lemma:T3:shift}, 
\mbox{$\chi\in K^\vartriangleright_{\underline a,3}$}.
Since $k< l$ and $g_k > 1$, one has that $a_l - a_k \geq 2$. Thus, one 
obtains that 
\begin{align}
\chi(x_{n_0}) & = (3^{a_l-a_k}+ 2\rho)\eta_{a_k-1}(x_k) = 
(3^{a_l-a_k}+ 2\rho) \dfrac{1}{9} = 3^{a_l-a_k -2} + \dfrac {2\rho} 9 
\equiv_1 \dfrac {2\rho} 9 \\
\chi(x_{n_1}) & = (3^{a_l-a_k}+ 2\rho)\eta_{a_k-1}(x_l) =
(3^{a_l-a_k}+ 2\rho) \cdot 3^{a_k-a_l -2}  = 
\dfrac 1 9 + \dfrac{2\rho}{3^{a_l-a_k+2}}.
\end{align}
Therefore,
\begin{align} \label{eq:T3:chi-x}
\chi(x) = \chi(x_{n_0})+\rho\chi (x_{n_1}) + \chi(x^\prime) = 
\frac \rho 3 + \dfrac{2}{3^{a_l -a_k +2}} + \chi(x^\prime).
\end{align}
We claim that 
\begin{align} \label{eq:T3:chi-xprime}
|\chi(x^\prime)|\leq\dfrac{11}{8\cdot 3^4}. 
\end{align}
If $N=1$, then $x^\prime =0$, and (\ref{eq:T3:chi-xprime}) holds 
trivially, so we may assume 
that $N > 1$. Put $s=n_2$. Since $g_n >1$ for all $n\in \mathbb{N}$,
one has $a_{s+j} -a_s \geq 2j$, and thus
\begin{align}
\eta_{a_k-1}(x_{s+j}) = 3^{a_k - a_{s+j} -2} \leq
3^{a_k - a_{s} -2j -2}
\end{align}
for all $i \in \mathbb{N}$. Consequently, 
\begin{align}
|\chi(x^\prime)| & \leq \sum\limits_{i=2}^N |\chi(x_{n_i})| 
\leq \sum\limits_{j=s}^\infty |\chi(x_{j})| =
\sum\limits_{j=0}^\infty |\chi(x_{j+s})| = 
m \left(\sum\limits_{j=0}^\infty |\eta_{a_k-1}(x_{j+s})|\right) \\
& \leq 
m \left( \sum\limits_{j=0}^\infty  3^{a_k - a_{s} -2j -2} \right)
\leq \dfrac{m}{3^{a_s-a_k +2}}
\left(\sum\limits_{j=0}^\infty \dfrac{1}{3^{2j}} \right) = 
\dfrac{m}{3^{a_s-a_k +2}} \cdot \frac 9 8  = 
\dfrac{m}{8\cdot 3^{a_s-a_k}}, 
\end{align}
because $s=n_2 \leq n_i < n_{i+1}$ for all $i \in \mathbb{N}$.
Since $k < l < s$, $g_k > 1$, and $g_l > 1$, one has that
$a_s-a_l \geq 2$ and $a_s -a_k \geq 4$. Therefore,
\begin{align}
|\chi(x^\prime)| & \leq \dfrac{m}{8\cdot 3^{a_s-a_k}} 
\leq \dfrac{3^{a_l-a_k}+2}{8\cdot 3^{a_s-a_k}}=
\frac 1 8 \left(\dfrac{1}{3^{a_s -a _l}} 
+ \dfrac{2}{3^{a_s -a_k}} \right) \leq
\frac 1 8 \left(\frac{1}{3^2} + \frac{2}{3^4}  \right) = 
\dfrac{11}{8 \cdot 3^4},
\end{align}
as required. Hence, using the fact that $a_l -a_k \geq 2$, we obtain that
\begin{align}
\dfrac{2}{3^{a_l-a_k+2}} + |\chi(x^\prime)| \leq 
\dfrac{2}{3^4} + \dfrac{11}{8 \cdot 3^4} = 
\dfrac{27}{8\cdot 3^4} = \dfrac{1}{24}.
\end{align}
By (\ref{eq:T3:chi-x}), this implies that 
\mbox{$\chi(x) \not\in \mathbb{T}_+$},
contrary to the assumption that 
\mbox{$x\in Q_{\mathbb{T}}(K_{\underline a,3})$}.
Hence, $N=0$, and $x=x_{n_0} \in K_{\underline a,3}$, as desired. 
\end{proof}


\section{Sequences of the form $\boldsymbol{3^{a_n}}$ in 
$\boldsymbol{\mathbb{J}_3}$}

\label{sect:J3}

In this section, we present the proof of Theorem~\ref{thm:J3:main}. Recall 
that the Pontryagin dual $\widehat {\mathbb{J}_3}$ of $\mathbb{J}_3$ is 
the Pr\"ufer group $\mathbb{Z}(3^\infty)$.
For $k \in \mathbb{N}$, let $\zeta_k\colon \mathbb{J}_3 \rightarrow 
\mathbb{T}$ denote the continuous character defined by 
$\zeta_k(1)=3^{-(k+1)}$. For $m \in \mathbb{N}$,
put $J_m = \{ k \in \mathbb{N} \mid m\zeta_k \in 
L_{\underline a,3}^\triangleright \}$ and
$Q_m = \{m\zeta_ k \mid k \in J_m\}^\triangleleft$.

\begin{lemma} \label{lemma:J3:J12}
$J_1 = J_2 = \mathbb{N} \backslash \underline{a}$.
\end{lemma}

\begin{proof}
One has $m\zeta_k(y_n) = m \cdot 3^{a_n-k-1}$. Clearly,
$3^{i} \in \mathbb{T}_+$ if and only if $i \neq -1$, and
$2\cdot 3^{i} \in \mathbb{T}_+$ if and only if $i \neq -1$.
Thus, for $m=1,2$, $m\zeta_k(y_n) \in \mathbb{T}_+$ if and only if
$a_n-k-1 \neq -1$, or equivalently, $a_n \neq k$.
\end{proof}

\begin{theorem} \label{thm:J3:Q12}
If $g_n > 1$ for every $n \in \mathbb{N}$, then $Q_1 \cap Q_2 = \{
\sum\limits_{n=0}^\infty \varepsilon_n y_n \mid
(\forall n \in \mathbb{N}) (\varepsilon_n \in \{-1,0,1\})\}$.
\end{theorem}

\begin{proof}
Recall that every element $x \in \mathbb{J}_3$ can be written in the form
$x= \sum\limits_{i=0}^\infty c_i \cdot 3^i$, 
where $c_i \in \{-1,0,1\}$. For $x$ represented in this form, one has
\begin{align} \label{eq:J3:zetak}
\zeta_k(x) = \sum\limits_{i=0}^\infty c_i \cdot 3^{i-k-1} \equiv_1
\sum\limits_{i=0}^k c_i \cdot 3^{i-k-1} = 
\sum\limits_{i=1}^{k+1} \frac{c_{k-i+1}}{3^{i}}.
\end{align}
($\subseteq$)
If $x \in Q_1 \cap Q_2$, then for every $k \in J_1=J_2$, the 
element $y=\zeta_k(x)$ satisfies $y \in \mathbb{T}_+$ and
$2y \in \mathbb{T}_+$. Thus, by Lemma~\ref{lemma:T3:b1}, the coefficient 
of $\frac 1 3$ in (\ref{eq:J3:zetak}) is zero, that is, $c_{k}=0$.
Since $J_1 = J_2 = \mathbb{N}\backslash \underline a$ holds
by Lemma~\ref{lemma:J3:J12}, we have shown that
if $c_i\neq 0$, then $i \not \in \mathbb{N}\backslash \underline a$, 
so $i \in \underline a$, and therefore $i = a_n$ for some 
$n \in\mathbb{N}$. Hence, $x$ has the form
\begin{align}
x= \sum\limits_{n=0}^\infty c_{a_n} \cdot 3^{a_n} = 
\sum\limits_{n=0}^\infty \varepsilon_n y_n,
\end{align}
where $\varepsilon_n = c_{a_n} \in \{-1,0,1\}$.

($\supseteq$) Let $k \in J_1 = J_2 = \mathbb{N}\backslash \underline a$
(see Lemma~\ref{lemma:J3:J12}), 
and $x=\sum\limits_{n=0}^\infty \varepsilon_n y_n$. Let $n_0$ denote the 
largest index such that $\zeta_k(y_{n_0}) \not \equiv_1 0$.
Then $a_{n_0} \leq k$, and since $k \not \in \underline a$,
one has $a_{n_0} < k$. Consequently, 
$\zeta_k(y_{n_0}) \leq \frac 1 9$. Thus,
$\zeta_k(y_{n_0 -i}) \leq \frac 1 {9^{i+1}}$, because 
$g_n > 1$ implies that $a_{n_0} - a_{n_0 -i} \geq 2i$. Therefore,
\begin{align}
\left|\sum\limits_{n=0}^{n_0} 
\varepsilon_n \zeta_k(y_n)\right| \leq
\sum\limits_{n=0}^{n_0} |\zeta_k(y_n)| = 
\sum\limits_{i=0}^{n_0} |\zeta_k(y_{n_0-i})| \leq
\sum\limits_{i=0}^{n_0}  \frac 1 {9^{i+1}} <
\sum\limits_{i=0}^{\infty}  \frac 1 {9^{i+1}} =
\frac 1 9 \cdot \frac 9 8 = \frac 1 8.
\end{align}
So, 
\begin{align}
\zeta_k(x)  = \sum\limits_{n=0}^{n_0} \varepsilon_n \zeta_k(y_n)
\in \mathbb{T}_2.
\end{align}
Hence, $\zeta_k(x) \in \mathbb{T}_+$ and  $2\zeta_k(x) \in \mathbb{T}_+$, 
as desired.
\end{proof}

\begin{lemma} \label{lemma:J3:shift}
If $g_n > 1$ for all $n \in \mathbb{N}$ and
$0\leq k<l$ in $\mathbb{N}$, then $a_l+1 \in J_m$ for $m=3^{a_l-a_k}\pm 
2$,  that is, $m \zeta_{a_l+1}\in L_{\underline a,3}^\triangleright$.
\end{lemma}

\begin{proof}
Let $\chi= m \zeta_{a_l+1}$ and $n\in \mathbb{N}$. 
If $l<n$, then $a_l+2\leq a_n$ (as $g_l >1$), and so 
\begin{align}
\chi(y_n) = m \zeta_{a_l+1}(y_n) = m \cdot 3^{a_n-a_l -2} \equiv_1 0.
\end{align}
If $k<n \leq l$, then $a_k+2 \leq a_n$ (since $g_k >1$), so 
$a_l-a_n+2 \leq a_l-a_k$,
and thus
\begin{align}
\chi(y_n) = m \zeta_{a_l+1}(y_n) = m \cdot 3^{a_n-a_l -2} = 
\frac m {3^{a_l -a_n+2}} = 
\frac{3^{a_l-a_k}\pm 2}{3^{a_l -a_n+2}} \equiv_1
\pm \frac{2}{3^{a_l -a_n+2}} \in \mathbb{T}_+,
\end{align}
because $a_l -a_n+2 \geq 2$. Finally, if $0 \leq n \leq k$, then
\begin{align}
\chi(y_n) = \frac{3^{a_l-a_k}\pm 2}{3^{a_l -a_n+2}} = 
\frac{1}{3^{a_k-a_n +2}} \pm \frac{2}{3^{a_l -a_n+2}}.
\end{align}
Since $n <l$ and $g_n >1$, one has $a_l - a_n \geq 2$, and so
$\frac 2 {3^{a_l -a_n+2}} \in [0,\frac 2 {81}]$. Furthermore,
$k \geq n$ implies that $a_k \geq a_n$, and so 
$\frac 1 {3^{a_k -a_n+2}} \in [0,\frac 1 {9}]$. Therefore,
$\chi(y_n) \in [0,\frac{11}{81}]\subseteq \mathbb{T}_+$.
This shows that 
$\chi =m\zeta_{a_l+1} \in L_{\underline a,3}^\triangleright$.
\end{proof}

\begin{proof}[Proof of Theorem~\ref{thm:J3:main}.]
Suppose that there is $n_0 \in \mathbb{N}$ such that
$g_{n_0}=1$. Then $y_{n_0+1}=3y_{n_0}$, and thus by 
Corollary~\ref{cor:nec:J-2x},
$2y_{n_0} \in Q_{\mathbb{J}_3}(\{y_{n_0},3y_{n_0}\}) \subseteq
Q_{\mathbb{J}_3}(L_{\underline{a},3})$. Since $2y_{n_0}\not\in
L_{\underline{a},3}$, this shows that $L_{\underline{a},3}$ is not 
quasi-convex.

Conversely, 
suppose that \mbox{$g_n > 1$} for every \mbox{$n\in \mathbb{N}$}, and let 
\mbox{$x \in Q_{\mathbb{J}_3}(L_{\underline{a},3})\backslash\{0\}$. }
We show that \mbox{$x \in L_{\underline{a},3}$.}
By Theorem~\ref{thm:J3:Q12}, $x$ can be expressed in the form
$x=\sum\limits_{n=0}^\infty \varepsilon_n y_n$, where
$\varepsilon_n\in\{-1,0,1\}$, because
\mbox{$Q_{\mathbb{J}_3}(L_{\underline{a},3}) \subseteq 
Q_1 \hspace{-2pt}\cap\hspace{-1pt} Q_2$}.
By discarding all summands with 
\mbox{$\varepsilon_n=0$}, we obtain that
\mbox{$x\hspace{-1.1pt}= \hspace{-1.9pt}
\sum\limits_{i=0}^N \varepsilon_{n_i} y_{n_i}$}, where 
$N \in \mathbb{N}  \cup\{\infty\}$, and
$\varepsilon_{n_i} \in \{-1,1\}$ and $n_i < n_{i+1}$ for every $i$.
Assume that $N >0$. By replacing $x$ with $-x$ if necessary, we may assume 
that $\varepsilon_{n_0}=1$. We put $\rho=\varepsilon_{n_1}$. Since
$n_2 \leq n_i$ for every $i \geq 2$, $y_{n_2}$ divides $y_{n_i}$ for every 
$i \geq 2$. Thus, $x=y_{n_0}+\rho y_{n_1} + y_{n_2}z$, where
$z \in \mathbb{J}_3$. Put $k=n_0$, $l=n_1$, and
$\chi=(\rho 3^{a_l-a_k}+2) \zeta_{a_l+1}$. By Lemma~\ref{lemma:J3:shift}, 
$\chi \in L_{\underline a,3}^\triangleright$, because
$\chi = \rho (3^{a_l-a_k}+2\rho) \zeta_{a_l+1}$ and $\rho=\pm 1$.
One has $\zeta_{a_l+1}(y_{n_2}) = 3^{a_{n_2}-a_l -2} \equiv_1 0$, because
$n_2 > l$ and $g_l > 1$, and consequently $a_{n_2}-a_l \geq 2$. So,
$\zeta_{a_l+1}(y_{n_2}z) =0$ in $\mathbb{T}$, because 
$\ker \zeta_{a_l+1} = 3^{a_l+2}\mathbb{J}_3$ is an ideal in the 
topological ring $\mathbb{J}_3$. Therefore,
$\chi(y_{n_2}z)=0$, and hence
\begin{align}
\chi(x) = \chi(y_{n_0}) + \rho  \chi(y_{n_1}).
\end{align}
Since $k <l$ and $g_k >1$, one has that $a_l-a_k \geq 2$. Thus,
one obtains that
\begin{align}
\chi(y_{n_0}) &= (\rho 3^{a_l-a_k} +2) 3^{a_k -a _l -2} = 
\dfrac{\rho}{9} + \frac{2}{3^{a_l -a_k +2}}, \quad\mbox{and} \\
\chi(y_{n_1}) &= (\rho 3^{a_l-a_k} +2) 3^{-2} = 
\rho 3^{a_l-a_k-2} + 2 \cdot 3^{-2} \equiv_1 \frac{2}{9}.
\end{align}
Therefore,
\begin{align}
\chi(x) = \chi(y_{n_0}) + \rho  \chi(y_{n_1}) = 
\frac \rho 3 +  \frac{2}{3^{a_l -a_k +2}}.
\end{align}
Since $a_l-a_k \geq 2$, one has $a_l-a_k +2\geq 4$, and consequently
$ \frac{2}{3^{a_l -a_k +2}} \leq \frac{2}{81}$. So,
$\chi(x) \not\in\mathbb{T}_+$, contrary to the assumption that
$x \in Q_{\mathbb{J}_3}(L_{\underline a,3})$. Hence, $N=0$, and
$x=y_{n_0} \in L_{\underline a,3}$, as desired.
\end{proof}

\section*{Acknowledgments}

We are grateful to Karen Kipper for her kind help in proof-reading this
paper for grammar and punctuation.

\nocite{Aussenhofer}   
\nocite{Banasz}        
\nocite{BeiLeoDikSte}  
\nocite{BrugMar2}      
\nocite{BruPhD}        
\nocite{deLeoPhD}      
\nocite{DikLeo}        
\nocite{Hern2}         
\nocite{Pontr}         
\nocite{Vilenkin}      

{\footnotesize

\bibliography{notes,notes2,notes3}
}

\begin{samepage}

\bigskip
\noindent
\begin{tabular}{l @{\hspace{1.8cm}} l}
Department of Mathematics and Computer Science & Department of Mathematics\\
University of Udine & University of Manitoba\\
Via delle Scienze, 208 -- Loc. Rizzi, 33100 Udine
 & Winnipeg, Manitoba, R3T 2N2 \\
Italy & Canada \\ & \\
\em e-mail: dikranja@dimi.uniud.it  &
\em e-mail: lukacs@cc.umanitoba.ca
\end{tabular}

\end{samepage}

\end{document}